\documentclass[11pt]{amsart}
\usepackage{amsfonts,amsmath,amsthm,mathrsfs}
 \newtheorem{thm}{Theorem}[section]
 \newtheorem{cor}[thm]{Corollary}
 \newtheorem{lem}[thm]{Lemma}{\rm}
 \newtheorem{prop}[thm]{Proposition}

 {\rm}
 
 \newtheorem{rem}[thm]{Remark}
 \numberwithin{equation}{section}

\def\x{\mathbf{x}}

\def\A{\mathbf{A}}

\def\H{\mathbf{H}}
\def\X{\mathbf{X}}

\def\v{\mathbf{v}}
\def\x{\mathbf{x}}

\def\N{\mathbb{N}}
\def\R{\mathbb{R}}

\def\om{\mathbf{\Omega}}
\def\bphi{\boldsymbol{\phi}}

\def\lf{\underline{f}}
\def\uf{\overline{f}}

\def\utheta{\overline{\theta}}
\def\otheta{\underline{\theta}}
\def\om{\mathbf{\Omega}}

\begin{document}
\title[Connecting optimization with tri-diagonal moment matrices]{Connecting optimization with spectral analysis of tri-diagonal matrices}
\thanks{This work was supported by the European Research Council (ERC)
under the ERC-Advanced Grant for the TAMING project, and also by
the AI Interdisciplinary Institute ANITI through the French ``Investing for the Future PIA3'' program under the Grant agreement n$^{\circ}$ANR-19-PI3A-0004.
}
\author{Jean B. Lasserre}
\begin{abstract}
We  show that the
global minimum (resp. maximum) of a continuous function on a compact set can be approximated from above (resp. from below) 
by computing the smallest (rest. largest) eigenvalue of a hierarchy
of $(r\times r)$ tri-diagonal matrices of increasing size. Equivalently it reduces to computing
the smallest (resp. largest) root of a certain univariate degree-$r$ orthonormal polynomial.
This provides a strong connection 
between the fields of optimization, orthogonal polynomials, numerical analysis and linear algebra, via asymptotic 
spectral analysis of tri-diagonal symmetric matrices. 
 \end{abstract}
\maketitle

\section{Introduction}
\label{section-1}
The goal of this paper is show that the global minimum (resp. maximum) of a continuous function $f$ on a compact set $\om\subset\R^n$
can be approximated as closely as desired from above (resp. from below) by the smallest 
eigenvalues $(\tau^\ell_r)_{r\in\N}$ (resp. largest eigenvalues $(\tau^u_r)_{r\in\N}$) of a sequence of tri-diagonal  univariate ``moment" matrices of 
increasing size $r$. Equivalently, $\tau^\ell_r$  (resp. $\tau^u_r$) is
the smallest (resp. largest) root of a univariate orthogonal polynomial of increasing degree $r$.
Thus it reveals a (perhaps suprising) connection between the fields
of optimization and the asymptotic spectral analysis of tri-diagonal symmetric matrices
(also related to the asymptotic analysis of univariate orthogonal polynomials). 

Notice that there is a large body of  literature on numerical analysis of tri-diagonal
symmetric matrices for which efficient  specialized algorithms exist 
(for instance the characteristic polynomial can be computed efficiently and roots of univariate polynomials
can also be computed efficiently); see e.g. \cite{ford, routh}. 

Our result is valid in a quite general context, namely if $f$ is a continuous function and $\om$ is compact.
However to turn it into a \emph{practical} algorithm requires effective computation of  such matrices. 
This in turn requires computing integrals 
$\int_\om f(\x)^kd\lambda$ for arbitrary $k$, where $\lambda$ is a measure 
whose support is exactly $\om$ (and whose choice is left to the user). If  
$f$ is a polynomial, $\om$ is a ``simple" compact set like a box, an ellipsoid, a simplex, a sphere, the discrete hypercube $\{0,1\}^n$
(or their image by an affine mapping) and $\lambda$ is a distinguished measure
(e.g., Lebesgue measure (or other families of polynomial densities), rotation invariant (for the sphere), counting measure for 
$\{0,1\}^n$) then such integrals are obtained in closed-form.

In addition, after the initial version of this paper was released, Laurent and Slot \cite{laurent-slot} proved 
a $O(\log^2 r/r^2)$ rate of convergence for the monotone sequence of upper bounds $(\tau^\ell_r)_{r\in\N}$ on the minimum (resp. lower bounds $(\tau^u_r)_{r\in\N}$ on the maximum)  under a weak geometric condition on $\om$ (satisfied by the above simple sets).

\subsection*{Background} Let $f:\om\to\R$ be a continuous function on a compact set $\om\subset\R^n$,
and consider the optimization problem:
\begin{equation}
\label{pb-opt}
\lf\,=\,\displaystyle\inf_\x\,\{\,f(\x):\,\x\in\om\,\};\quad
\uf\,=\,\displaystyle\sup_\x\,\{\,f(\x):\,\x\in\om\,\}.
\end{equation}

In \cite{lass-siopt} the author showed that one may approximate $\lf$ from above 
(resp. $\uf$ from below) by solving the following hierarchy\footnote{The hierarchy 
\eqref{pb-initial-inf} should not be confused with the Moment-SOS hierarchy described in \cite{lass-book} 
to solve the same problem \eqref{pb-opt}, but which yields a converging sequence of {\em lower bounds} $\rho_r\uparrow \lf$ as $r\to\infty$.} of
optimization problems indexed by $r\in\N$:
\begin{eqnarray}
\label{pb-initial-inf}
\otheta_r\:\,=\,\inf_{\sigma}\:\{\,\int f\,\sigma\,d\lambda:\:\int \sigma\,d\lambda\,=\,1; \quad \sigma\in\Sigma[\x]_r\,\}\\
\label{pb-initial-sup}
\utheta_r\:\,=\,\sup_{\sigma}\:\{\,\int f\,\sigma\,d\lambda:\:\int \sigma\,d\lambda\,=\,1; \quad \sigma\in\Sigma[\x]_r\,\},
\end{eqnarray}
where $\lambda$ is a fixed measure whose support is exactly $\om$ and $\Sigma[\x]_r$ is the convex cone of SOS polynomials of degree at most $2r$. Indeed $\otheta_r\downarrow\lf$ (resp. $\utheta_r\uparrow\uf$) as $t\to\infty$. 
If $f$ is a polynomial and one knows all moments of the measure $\lambda$ on $\om$,
then each problem \eqref{pb-initial-inf} is a very specific {\em semidefinite program}.
As a matter of fact it reduces to a generalized eigenvalue problem involving two moment-like 
matrices whose size ${n+r\choose r}$ increases with $r$. For instance this is the case whenever 
$\om$ is a simple set (e.g. the box $[-1,1]^n$, the Euclidean unit ball, the sphere, the simplex, the discrete hypercube, and their affine transformation)
and $\lambda$ is an appropriate measure (e.g., Lebesgue measure on unit box, Euclidean unit ball, simplex, rotation invariant measure on the sphere, counting measure on the discrete hypercube, etc.).

In a recent series of papers \cite{las-1,las-2,las-3,las-4,las-5}, de Klerk, Laurent and their co-workers 
have been able to analyze 
the convergence $\otheta_r\downarrow\lf$ of such a hierarchy by appropriate clever choices of the reference measure (as indeed $\lambda$ it can be any measure whose support is exactly $\om$). 
Ultimately they could provide rates of convergence. In particular and remarkably, they 
show that for certain important sets and reference measures (e.g. the box $\om=[-1,1]^n$ and the sphere $\mathbb{S}^{n-1}$) a convergence rate $O(1/r^2)$ is achieved.

However all variants of \eqref{pb-initial-inf} and \eqref{pb-initial-sup} 
consider density polynomials $\sigma\in\R[\x]$
in $n$ variables which results in eigenvalue problems with multivariate Hankel-type matrices 
whose size ${n+r\choose n}$ grows very quickly with $r$.
Therefore and so far, this hierarchy of upper (resp. lower) bounds has not been proved to be efficient in practice (even for small  size problems) and its main interest is rather theoretical as it provides an algorithm with proven 
rate of convergence $O(1/r^2)$ to the global optimum on some simple sets.

\subsection*{Contribution}

Our contribution is twofold:\\

{\bf I.} We first provide an alternative converging  hierarchy of upper bounds in the same spirit as
\eqref{pb-initial-inf} and \eqref{pb-initial-sup} but following a different path. 
{\em The main distinguishing feature is to reduce the 
$n$-dimensional initial problem to a one-dimensional equivalent problem
by using the pushforward measure $\#\lambda$ of the measure $\lambda$ on $\om$,
by the mapping $f:\om\to\R$. It results in solving again a hierarchy of eigenvalue problems
but with the major advantage of considering Hankel moment matrices in just ONE variable (hence of size $r$
in contrast to ${n+r\choose n}$).}

To achieve this result we exploit the fact that $\lf$ and $\uf$ are the left and right endpoints 
of the support of $\#\lambda$ on the real line, and therefore
by invoking Lasserre \cite[Theorem 3.3]{lass-pams}, one may approximate $\lf$ from above by solving:
\begin{eqnarray}
\nonumber
\tau^\ell_r&=&\sup\,\{a:\, \H_r(x;\#\lambda)\,\succeq\,a\,\H_r(\#\lambda)\,\}\\
\label{practical}
&=&\lambda_{\min}(\H_r(x;\#\lambda),\H_r(\#\lambda)),\qquad \forall r\in\N,
\end{eqnarray}
where the real symmetric 
$(r+1)\times (r+1)$ matrix $\H_r(\#\lambda)$ (resp. $\H_r(x;\#\lambda)$)
is the moment matrix of the pushforward $\#\lambda$ (resp. 
the localizing matrix associated with $\#\lambda$ and the univariate linear polynomial $x\mapsto x$).
Similarly, one may approximate $\uf$ from below by solving:
\begin{eqnarray*}
\tau^u_r&=&\inf\,\{a:\, a\,\H_r(\#\lambda)\,\succeq\,\H_r(x;\#\lambda)\,\}\\
&=&\lambda_{\max}(\H_r(x;\#\lambda),\H_r(\#\lambda)),\qquad \forall r\in\N,
\end{eqnarray*}
and indeed $\tau^\ell_r\downarrow\lf$ (resp. $\tau^u_r\uparrow \uf$) as $r$ increases; see \cite{lass-pams}.

\begin{rem}
{\rm Equivalently, by duality in convex optimization:
\begin{equation}
\label{aux1}
\tau^\ell_r\,=\,\inf_\sigma\,\{\, \int z\,\sigma\,d\#\lambda:\:\int \sigma\,d\#\lambda=1;\quad\sigma\in\Sigma[z]_r\,\},
\end{equation}
where $\Sigma[z]_r$ is the convex cone of {\em univariate} SOS polynomials of degree at most $2r$; similarly
for $\tau^u_r$ just replace ``$\inf$" by ``$\sup$".
The formulation \eqref{aux1} resembles \eqref{pb-initial-inf} but
with the important difference that in \eqref{pb-initial-inf} one searches over SOS polynomials
of degree at most $r$  in $n$-variables whereas in \eqref{aux1} one searches over SOS polynomials of degree at most $r$ in ONE variable. 

Also notice  that in contrast to \eqref{pb-initial-inf}, in \eqref{aux1}  (or in \eqref{practical}) the function $f$ to minimize 
does {\em not} appear explicitly; it is encoded in the pushforward measure $\#\lambda$.
So if one is able to compute (or approximate) the moments of $\#\lambda$, then 
both matrices $\H_r(x;\,\#\lambda)$ and $\H_r(\#\lambda)$ are known and 
\eqref{practical} can be solved in practice. Also to solve \eqref{pb-initial-inf} in practice one needs $f$
to be a polynomial.
}\end{rem}

{\bf II.} We next further simplify the analysis by considering an orthonormal basis $(T_j)_{j\in\N}$ of polynomials w.r.t.
the pushforward measure $\#\lambda$. Recall that a family of $T_j$'s, $j\leq r$, can be obtained from the moment matrix 
$\H_r(\#\lambda)$ by simple determinant computations.
In this new basis $(T_j)_{j\in\N}$, the moment matrix $\hat{\H}_r(\#\lambda)$ becomes the identity
and the localizing matrix $\hat{\H}_r(x;\,\#\lambda)$ now becomes a {\em tri-diagonal} (banded) symmetric matrix
whose coefficients have a direct expression in terms of the parameters defining the classical three-term recurrence relation 
satisfied by the $T_j$'s. 

{\em Therefore the convergence $\tau^\ell_r\downarrow \lf$ (resp. $\tau^u_r\uparrow \uf$) is simply  
the asymptotic behavior of the smallest (resp. largest) eigenvalue of tri-diagonal $r\times r$ 
``moment" matrices, as $r$ increases; equivalently the asymptotic behavior of the smallest
(rest. largest) root of a certain univariate polynomial orthogonal w.r.t. $\#\lambda$.}
 
This reveals a strong and perhaps surprising connection between the fields of (global) 
optimization and the spectral analysis of tri-diagonal Hankel matrices (or extremal roots of a family of orthogonal polynomials). 
Actually such a link already appeared in de Klerk and Laurent \cite{0} to analyze convergence of upper bounds \eqref{pb-initial-inf}
for the specific univariate (trivial) optimization problem $\min\,\{x: x\in [-1,1]\}$ and specific reference measure $\lambda$. Then they used 
this univariate problem as  a building block to prove rates of convergence of the bounds in \eqref{pb-initial-inf}
in case where $\om=[-1,1]^n$ or $\mathbb{S}^{n-1}$ (and with specific measures $\lambda$).

There is a large body of literature on various aspects of tri-diagonal symmetric matrices, including practical efficient algorithms;
see for instance Aurentz et al. \cite{aurentz}, Businger \cite{businger}, Ford \cite{ford}, Kili\c{c} \cite{kilos}, 
Mallik \cite{mallik}, Osipov \cite{osipov}, and Routh \cite{routh}.
Therefore this may also open the door to practical algorithms for good approximations of $\lf$ and $\uf$
in non trivial cases, as soon as one can obtain moments of the measure $\#\lambda$ for reasonably large degrees.

If the sequence  $(\tau^\ell_r)_{r\in\N}$ (or $(\tau^u_r)_{r\in\N}$) has 
obvious numerical advantages when compared with $(\otheta_r)_{r\in\N}$ (or $(\utheta_r)_{r\in\N}$)
for a same fixed $r$, the main drawback of $(\tau^\ell_r)_{r\in\N}$ is the computation
of moments of $\#\lambda$ which may become tedious for non modest dimension $n$.
However, {\em sparsity} of $f$ (i.e. when $f$ has a few monomials only) can be exploited. 
Also the rate of convergence for $\tau^\ell_r\downarrow \lf$ as $r\to\infty$,
is more difficult to analyze because we do not know the density of the pushforward measure $\#\lambda$ with respect to Lebesgue measure on $\om$.

Actually, recently Laurent and Slot \cite{laurent-slot}
have provided a further detailed analysis of some relative merits of the sequences
$(\tau^\ell_r)_{r\in\N}$ and $(\otheta_r)_{r\in\N}$ beyond the scope the present paper. As already mentioned, 
they prove that remarkably, $\tau^\ell_r\downarrow \lf$ at a $O(\log^2 r/r^2)$ rate under a weak geometric condition on $\om$
(satisfied for the simple sets $\om$ mentioned previously).

\section{Main result}

\subsection{Notation, definitions and preliminary results}
\label{definitions}
Let $\R[\x]$ denote the ring of polynomials in the $n$ variables $\x=(x_1,\ldots,x_n)$
and $\R[\x]_t\subset\R[\x]$ denote the vector space of polynomials  of degree at most $t$, hence of dimension $s(t)={n+t\choose n}$.
Let $\Sigma[\x]\subset\R[\x]$ denote the space of polynomials that are sums-of-squares (in short SOS polynomials) and
let $\Sigma[\x]_t\subset\R[\x]_{2t}$ denote the space of SOS polynomials of degree at most $2t$.
For univariate polynomials in the variable $x$, we use the notation $\R[x]$, $\Sigma[x]$, $\R[x]_t$ and $\Sigma[x]_t$.

With $\alpha\in\N^n$ and $\x\in\R^n$, the notation $\x^\alpha$ stands for $x_1^{\alpha_1}\cdots x_n^{\alpha_n}$. Also for every $\alpha\in\N^n$, let $\vert\alpha\vert:=\sum_i\alpha_i$
and $\N^n_t:=\{\alpha\in\N^n:\vert\alpha\vert\leq t\}$, where $\N=\{0,1,2,\ldots\}$.

The support of a Borel measure $\mu$ on $\R^n$ is the smallest closed set $\om$ such that 
$\mu(\R^n\setminus\om)=0$. Denote by $\mathcal{B}(\X)$ the Borel $\sigma$-field associated with a topological space $\X$, and $\mathscr{M}(\X)_+$ the space of finite nonnegative Borel measures on $\X$.

\subsection*{Generalized eigenvalue} Given two real symmetric matrices $\A,\mathbf{C}\in\R^{n\times n}$ denote by 
$\lambda_{\min}(\A,\mathbf{C})$ 
the smallest generalized eigenvalue with respect to the pair $(\A,\mathbf{C})$, that is, the smallest scalar $\tau$ such that $\A\x=\tau\,\mathbf{C}\x$ for some nonzero vector $\x\in\R^n$.
When $\mathbf{C}$ is the identity matrix then $\lambda_{\min}(\A,\mathbf{C})$ is just the smallest eigenvalue of $\A$. Computing $\lambda_{\min}(\A,\mathbf{C})$ can be done via a pure and efficient linear algebra routine.
The notation $\A\succeq0$ (resp. $\A\succ0$) stands for $\A$ is positive semidefinite (resp. positive definite). 
If $\A,\mathbf{C}\succ0$ then:
\begin{equation}
\label{gen-eig-sdp}
 \lambda_{\min}(\A,\mathbf{C})\,=\,\sup_\tau\{\,\tau: \A\succeq\tau \,\mathbf{C}\,\}.
 \end{equation}

\subsection*{Moment matrix}
Given a real sequence $\bphi=(\phi_\alpha)_{\alpha\in\N^n}$, let
$\H_r(\bphi)$ denote the multivariate (Hankel-type) moment matrix defined by $\H_r(\bphi)(\alpha,\beta)=\phi_{\alpha+\beta}$ for all $\alpha,\beta\in\N^n_r$. 
For instance, in the univariate case $n=1$, with $r=2$, $\H_2(\bphi)$ is the Hankel matrix
\[\H_2(\bphi)\,=\,\left[\begin{array}{ccc} \phi_0 &\phi_1& \phi_2\\
\phi_1 &\phi_2& \phi_3\\
\phi_2 &\phi_3 &\phi_4\end{array}\right].\]
If $\bphi=(\phi_j)_{j\in\N}$ is the moment sequence of a Borel 
measure $\phi$ on $\R$ then $\H_r(\bphi)\succeq0$ for all $r=0,1,\ldots$. Conversely,
if $\H_r(\bphi)\succeq0$ for all $r\in\N$, then $\bphi$ is the moment sequence of some finite (nonnegative) Borel measure $\phi$ on $\R$. The converse result is not true anymore in the multivariate case.

\subsection*{Localizing matrix}
Given a real sequence $\bphi=(\phi_\alpha)_{\alpha\in\N^n}$ and $g\in\R[\x]$, $\x\mapsto g(\x)=\sum_\alpha g_\alpha\,\x^\alpha$,
$\H_r(g\,;\bphi)$ denote the multivariate (Hankel-type) localizing matrix defined by 
\[\H_r(g\,;\bphi)(\alpha,\beta)\,=\,\sum_\gamma g_\gamma\,\phi_{\gamma+\alpha+\beta},\quad \forall \alpha,\beta\in\N^n_.\]
For instance, in the univariate case $n=1$, with $r=2$ and $g(x)=x$, $\H_2(g\,;\bphi)$ is the Hankel matrix
\[\H_2(g\,;\bphi)\,=\,\left[\begin{array}{ccc} \phi_1 &\phi_2& \phi_3\\
\phi_2 &\phi_3& \phi_4\\
\phi_3 &\phi_4 &\phi_5\end{array}\right].\]
Equivalently, if $\bphi$ is the moment sequence of a Borel 
measure $\phi$ on $\R^n$ then $\H_r(g\,;\bphi)$ is the moment matrix $\H_r(\boldsymbol{\nu})$ associated with the measure $d\nu=g\,d\phi$.

\subsection*{Pushforward measure}

Let $\lambda\in\mathscr{M}(\om)_+$ be 
a finite Borel measure on $\om\subset\R^n$ whose support is exactly $\om$, that is, $\om$ is 
the smallest closed set such that $\lambda(\R^n\setminus\om)=0$.
Let $\#\lambda$ be the {\em pushforward measure} on $\R$ of $\lambda$ with respect to (w.r.t.) the mapping
$f:\om\to \R$. That is:
\[\#\lambda(C)\,=\,\lambda(f^{-1}(C)),\qquad\forall C\in\mathcal{B}(\R).\]
In particular, its moments $\boldsymbol{\#\lambda}=(\#\lambda_k)_{k\in\N}$ read:
\begin{equation}
\label{mom-formula}
\#\lambda_k\,=\,\int_{[0+\infty)}z^k\,d\#\lambda(z)\,=\,\int_{\om}f(\x)^k\,d\lambda(\x),\qquad k=0,1,\ldots\end{equation}
It is straightforward to see that the support of $\#\lambda$ is contained in the interval $[\lf,\uf]$ with $\lf$ and $\uf$
as its left and right endpoints.
\begin{rem}[Computing the moments $\boldsymbol{\#\lambda}$]
\rm{While \eqref{mom-formula} is quite general, exact numerical computation of $\#\lambda_k$  is \emph{not} possible in full generality.
However, for specific combinations of ($f$, $\om$, $\lambda$) one may obtain $\boldsymbol{\#\lambda}$ in closed-form 
(let alone combinations where numerical approximations schemes, e.g. cubature formula or Monte-Carlo, can be used).
This is the case when $f$ is a polynomial and $\om$ is a ``simple" set like a box, an ellipsoid, a  simplex, a sphere, the discrete hypercube $\{0,1\}^n$,
or their image by an affine map. Then several choices of $\lambda$ are possible: 
Lebesgue measure (or specific families of polynomial densities) if $\om$
is a box, simplex or ellipsoid, the rotation invariant measure if $\om$ is the sphere, 
the discrete counting measure if $\om$ is the hyper cube, etc. 
For instance, if $\om=[-1,1]^n$ and $\lambda$ is Lebesgue measure then writing
\[f(\x)^k\,=\,\sum_{\alpha\in\N^n_{2kd}}f_{k\alpha}\,\x^\alpha,\quad k=0,1,\ldots,\]
one obtains
\[\#\lambda_k\,=\,\sum_{\alpha\in\N^n_{2kd}}f_{k\alpha}\,\displaystyle\prod_{i=1}^n \left(\int_{-1}^1x^{\alpha_i}\,dx\right),\quad k\in\N.\]
Similarly if $\om=\{\x:\Vert\x\Vert_2^2\leq1\,\}$ then for every $k\in\N$:
\begin{eqnarray*}
\#\lambda_k=\int_\om g(\x)^k\,d\lambda&=&\frac{1}{\Gamma(1+\frac{n+kt}{2})}\int_{\R^n}f(\x)^k\exp(-\Vert\x\Vert_2^2)\,d\lambda\\
&=&\frac{1}{\Gamma(1+\frac{n+kt}{2})}\sum_{\vert\alpha\vert=kt}f_{k\alpha}\prod_{i=1}^n
\int_\R x^{\alpha_i}\exp(-x^2)\,dx,
\end{eqnarray*}
which is also obtained in closed form. Notice that for large $k$, computing $\#\lambda_k$ can be time and space consuming.}
\end{rem}
\begin{rem}[On compactness assumption]
{\rm In fact non compact sets, e.g., 
$\om=\R^n$ with $\lambda$ being the normal distribution $\mathcal{N}(0,I)$, or $\om=\R^n_+$ with $\lambda$ being the exponential 
distribution $d\lambda=\prod_i\exp(-x_i)\,d\x$, can be treated as well; see 
\cite{lass-siopt}. 

In particular, and interestingly, it permits to provide upper bounds that converge to the global minimum $f^*$  of $f$ on $\R^n$, even if
$f^*$ is not attained (a difficult case in practice because then $\Vert\x\Vert\to\infty$ if $f(\x)\to f^*$). For instance 
with $f(x,y)=(xy-1)^2+y^2$ on $\R^2$, $f^*=0$ is not attained and $f(n,1/n)\to 0$ as $n\to\infty$. Then letting
$d\lambda=(2\pi)^{-1}\exp(-(x^2-y^2)/2)dxdy$, the pushforward $\#\lambda$ is supported on $[0,+\infty)$, and
\[\#\lambda_k\,=\,\frac{1}{2\pi}\int_{\R^2} ((xy-1)^2+y^2)^k\exp(-x^2/2)\,\exp(-y^2/2)\,dxdy,\quad k\in\N,\]
is obtained in closed-form. However, for ease of exposition and to avoid technicalities, we restrict to the compact case.}
\end{rem}
\subsection{Main result}

We have just seen that the support of the pushforward $\#\lambda$ is precisely contained
in the interval $[\lf,\uf]$, with $\lf,\uf$ as its left and right endpoints.
Therefore the problem of approximating $\lf$ and $\uf$ reduces
to approximate the endpoints of the support of $\#\lambda$ from the sole knowledge of its moments. 
That is,
\begin{equation}
\label{equiv}
\lf\,=\,\min\,\{\,x: x\,\in\,{\rm supp}(\#\lambda)\,\};\quad 
\uf\,=\,\max\,\{\,x: x\,\in\,{\rm supp}(\#\lambda)\,\}.
\end{equation}
In \cite{lass-pams} the author has already considered the more general problem of 
bounding the support of a measure 
$\mu$ on $\R^n$ from knowledge of its marginal moments. In our case 
$\mu$ is the push forward measure $\#\lambda$ on $\R$ and therefore we can invoke Theorem 3.3 in \cite{lass-pams}. More precisely:

Let $\om\subset\R^n$ be compact with nonempty interior, $\lambda\in\mathscr{M}(\om)_+$
and consider the hierarchy of optimization problems indexed by $r\in\N$:
\begin{eqnarray}
\label{relax-inf}
\tau^\ell_r&=&\sup_a\,\{\,a:\:\H_r(x;\,\#\lambda)\,\succeq\,a\,\H_r(\#\lambda)\,\}\\
\label{relax-sup}
\tau^u_r&=&\inf_a\,\{\,a:\:a\,\H_r(\#\lambda)\,\succeq\,\H_r(x;\,\#\lambda)\,\},
\end{eqnarray}
where $\H_r(x;\,\#\lambda)$ is the (univariate) localizing matrix associated with the polynomial $x\mapsto x$ and the measure $\#\lambda$ on $\R$, and
$\H_r(\#\lambda)$ is the (univariate) moment matrix associated with $\#\lambda$.

\begin{thm}
\label{th1}
Let $\lambda\in\mathscr{M}(\om)_+$ be such that ${\rm supp}(\lambda)=\om$. Then:
\begin{equation}
\label{th1-1}
\lf\leq \tau^\ell_r\,\leq\,\tau^u_r\,\leq\,\uf \quad \forall r\in\N.
\end{equation}
In addition the sequence $(\tau_r^\ell)_{r\in\N}$ (resp. $(\tau_r^u)_{r\in\N}$)
is monotone non-increasing (resp. non-decreasing), and:
\begin{equation}
\label{th1-2}
\lim_{r\to\infty}\tau^\ell_r\,=\,\lf\,;\quad\lim_{r\to\infty}\tau^u_r\,=\,\uf.
\end{equation}
Finally, for all $r\in\N$:
\begin{eqnarray}
\label{th1-3}
\tau^\ell_r&=&\lambda_{\min}(\H_r(x;\,\#\lambda),\H_r(\#\lambda))\\
\label{th1-4}
\tau^u_r&=&\lambda_{\max}(\H_r(x;\,\#\lambda),\H_r(\#\lambda)).
\end{eqnarray}
\end{thm}
\begin{proof}
That the sequence $(\tau^\ell_r)_{r\in\N}$ is monotone non-increasing is straightforward as
the feasible set in \eqref{relax-inf} shrinks with $r$. The same argument 
shows that $(\tau^u_r)_{r\in\N}$ is monotone non-decreasing. Next, the support
of $\#\lambda$  is contained in the interval $[\lf,\uf]$ with $\lf$ and $\uf$
as its left and right endpoints, and we know all moments of $\#\lambda$. Therefore 
\eqref{th1-1}-\eqref{th1-3} follow from \cite[Theorem 3.3, p. 3379]{lass-pams}.
\end{proof}

Even though $\H_r(x;\,\#\lambda)$ may not be positive definite
we still have
\[\sup_{a}\,\{\,a:\:\H_r(x;\,\#\lambda)\,\succeq\,a\,\H_r(\#\lambda)\,\}\,=\,
\lambda_{\min}(\H_r(x;\,\#\lambda),\H_r(\#\lambda)).\]
This is because $\H_r(x-a;\,\#\lambda)=\H_r(x;\,\#\lambda)-a\,\H_r(\#\lambda)$; see Lemma \ref{lem-aux}.\\

So after one has reduced the $n$-dimensional problem \eqref{pb-initial-inf}
to the one-dimensional problem \eqref{aux1}, Theorem \ref{th1} shows that one thus has to handle
Hankel moment matrices of size $r$ only whereas $\otheta_r$ requires to handle
Hankel-like moment matrices of size ${n+r\choose r}$. However the moment {\em information}
needed to build up the moment and localizing matrices $\H_r(\#\lambda)$, $\H_r(x;\,\#\lambda)$,
still requires computing 
$\int_\om f(\x)^{2k}d\lambda$ with $k\leq 2r$ (hence handling $n$-variate moments up to degree $2d$).
But once this has been done, the eigenvalue problem is much easier.

\subsection{On finite convergence}

For same reasons as for the convergence $\underline{\theta}_r\downarrow\lf$, in \eqref{th1-2} 
the convergence $\tau^\ell_r\downarrow\lf$ 
as $r\to\infty$, is only \emph{asymptotic} and \emph{not} finite in general (and similarly for $\tau^u_r\uparrow\lf$ in \eqref{th1-4}). 
Indeed the rationale behind \eqref{aux1}  is to approximate the Dirac measure 
$\delta_{z^*}$ at the minimizer $z^*=\lf\in\R$ by a measure of the form $\sigma d\#\lambda$,
and such an approximation cannot be exact because the pushforward measure 
$\#\lambda$ has no atom if $\om$ is not finite.

An exception is when $\om$ is a finite set (e.g., the discrete hypercube $\{0,1\}^n$ or $\{-1,1\}^n$)
and $\lambda$ is the counting measure 
$\sum_{\v\in\om} \delta_\v$, with $\delta_\v$ being the Dirac  measure at $\v\in\om$.
Introduce the finite set $\mathcal{Z}:=\{f(\v): \v\in\om\}\subset\R$.
\begin{lem} 
\label{lem-finite}
Let $\om\subset\R^n$ be a finite set and let $\lambda$ be the counting measure on $\om$.
Then finite convergence takes place at most at $r=\bar{r}=\mathrm{card}(\mathcal{Z})-1$. That is, 
$\tau^\ell_{\bar{r}}=\lf$ and $\tau^u_{\bar{r}}=\uf$.
\end{lem}
\begin{proof}
The support of the pushforward measure $\#\lambda$ is the finite set $\mathcal{Z}$
with smallest element $z^*:=\lf\in\mathcal{Z}$ and largest element $\uf\in\mathcal{Z}$. 
In addition, the pushforward measure $\#\lambda$ satisfies $\#\lambda=\sum_{z\in\mathcal{Z}}a_z\,\delta_{z}$,
where:
\begin{eqnarray*}
0\,<\,a_z\,=\,\lambda(f^{-1}(\{z\}))&=&\lambda(\{\v\in\om: f(\v)=z\})\\
&=&\mathrm{card}(\{\v\in\om: f(\v)=z\}),\quad z\in\mathcal{Z}.
\end{eqnarray*}
Let $p\in\R[z]_{\bar{r}}$ be the degree-$\bar{r}$ polynomial 
$z\mapsto \prod_{a\in\mathcal{Z}; a\neq z^*} \frac{z-a}{z^*-a}$, so that
$p(z)=0$ for all $z^*\neq z\in\mathcal{Z}$, and $p(z^*)=1$.
Then letting $\sigma^*:=a_{z^*}^{-1}p^2\in\Sigma[z]_{\bar{r}}$, 
one obtains $\int\sigma^* d\#\lambda=\sum_{z\in\mathcal{Z}}a_z\,\sigma^*(z)=p(z^*)^2=1$, and 
$\int z\sigma^*\,d\#\lambda\,=\,z^*=\lf$. This proves that $\sigma^*$ is an optimal solution of \eqref{aux1}
when $r=\bar{r}$ (recall that $\tau^\ell_r\geq\lf$ for all $r$) and so $\tau^\ell_{\bar{r}}=\lf$ (and with a similar argument, $\tau^u_{\bar{r}}=\uf$).
\end{proof}
So Lemma \ref{lem-finite} states that $\bar{r}=\mathrm{card}(\mathcal{Z})-1$ determines when finite convergence is guaranteed,
i.e., $\tau^\ell_{\bar{r}}=\lf$. In other words, the larger is the set $\mathcal{Z}$ the more difficult is to
compute $\lf$ (or $\uf$).

\subsection{Link with tri-diagonal matrices}

Let $(T_j)_{j\in\N}$ be a family of orthonormal polynomials with respect to $\#\lambda$. 
For instance, such a family can be computed from the moments $(\#\lambda_k)_{k\in\N}$ as follows.
$T_0(x)=1=D_0(x)/\#\lambda_0$ for all $x\in\R$. Then compute the degree-one  polynomial:
\[x\mapsto D_1(x)={\rm det}\left[
\begin{array}{cc}
  \#\lambda_0& \#\lambda_1   \\
 1 & x\end{array}\right]\,=\,\#\lambda_0\,x-\#\lambda_1\]
and normalize $T_1(x)=a\,D_1(x)$ to obtain $\int T_1(x)^2d\#\lambda=1$, i.e.,
$T_1(x)=a\,D_1(x)$ with $a=\#\lambda_0^{-1/2}(\#\lambda_0\#\lambda_2-\#\lambda_1^2)^{-1/2}$. 
Next, to obtain $T_2$ compute
\[x\mapsto D_2(x)={\rm det}\left[
\begin{array}{ccc}
  \#\lambda_0& \#\lambda_1 &\#\lambda_2   \\
  \#\lambda_1& \#\lambda_2 &\#\lambda_3   \\
 1 & x &x^2\end{array}\right]\]
 and again normalize $T_2(x)=a\,D_2(x)$ to obtain $\int T_2^2d\#\lambda=1$, etc.
Next, the orthonormal polynomials satisfy the so-called {\em three-term recurrence} relation: 
\begin{equation}
\label{three-term}
x\,T_j(x)\,=\,a_j\,T_{j+1}(x)+b_j\,T_j(x)+a_{j-1}\,T_{j-1}(x),\quad\forall x\in\R,\quad j\in\N,\end{equation}
where $a_j=(d_j\,d_{j+2}/d_{j+1}^2)^{1/2}$, $b_j=\int x T_j(x)^2\,d\#\lambda$, and
\[d_j={\rm det}\left[
\begin{array}{ccccc}
  \#\lambda_0& \#\lambda_1 & \cdots &\cdots&\#\lambda_{j-1}   \\
  \#\lambda_1& \#\lambda_2 &\cdots &\cdots&\#\lambda_{j}   \\
 \cdots& \cdots &\cdots &\cdots&\cdots\\
 \#\lambda_{j-1} & \#\lambda_j&\dots&\cdots& \#\lambda_{2j-2}\end{array}\right],\quad j\in\N.\]
 The tri-diagonal infinite matrix:
 \begin{equation}
 \label{jacobi}
 J=\left[\begin{array}{cccccc}
 b_0 &a_0&0&\cdots&\cdots&0\\
 a_0&b_1 &a_1 &0 &\cdots&0\\
 0&a_1&b_2 &a_2 &\cdots &0\\
 0&0&\cdots&\cdots&\cdots&0\end{array}\right]\end{equation}
 is called the {\em Jacobi matrix} associated with the orthonormal polynomials $(T_j)_{j\in\N}$;
 see e.g. Dunkl and Xu \cite[pp. 10--11]{dunkl}.
 \begin{prop}
 \label{prop1}
Expressed in the basis of orthonormal polynomials $(T_j)_{j\in\N}$, the moment matrix $\hat{\H}_r(\#\lambda)$
is the identity matrix while the $(r+1)\times (r+1)$ localizing matrix $\hat{\H}_r(x;\,\#\lambda)$ is the $r$-truncation
\begin{equation}
\label{r-truncation}
J_r=\left[\begin{array}{cccccc}
 b_0 &a_0&0&\cdots&\cdots&0\\
 a_0&b_1 &a_1 &0 &\cdots&0\\
 0&a_1&b_2 &a_2 &\cdots &0\\
  0&\cdots&\cdots&\cdots&\cdots &0\\
 0&0&\cdots&a_{r-2}&b_{r-1}&a_{r-1}\\
 0&0&\cdots&\cdots&a_{r-1}&b_r
   \end{array}\right]\end{equation}
of the Jacobi matrix \eqref{jacobi}.
 \end{prop}
 \begin{proof}
 That the moment matrix $\hat{\H}_r(\#\lambda)$ expressed in the basis $(T_j)_{j\in\N}$ 
is the identity matrix follows from
\[\hat{\H}_r(\#\lambda)(i,j)\,=\,\int T_i(x)\,T_j(x)\,d\#\lambda\,=\,\delta_{i=j},\quad \forall i,j=0,1,\ldots,r\]
Next, in this basis the localizing matrix $\hat{\H}_r(x;\,\#\lambda)$ 
associated with $\#\lambda$ and the polynomial $x\mapsto x$, reads:
 \begin{eqnarray*}
 \hat{\H}_r(x;\,\#\lambda)(i,j)&=&\int x\,T_i(x)\,T_j(x)\,d\#\lambda\\
 &=&\underbrace{a_i\,\int T_{i+1}(x)\,T_j(x)\,d\#\lambda}_{=a_i\,\delta_{i+1=j}}
+ \underbrace{b_{i}\,\int T_{i}(x)T_j(x)\,d\#\lambda}_{b_i\,\delta_{i=j}}\\
&&+\underbrace{a_{i-1}\,\int T_{i-1}(x)T_j(x)\,d\#\lambda}_{a_{i-1}\,\delta_{i-1=j}}\\
&=&\mbox{$0$ if $j\not\in\{i-1,i,i+1\}$,}\end{eqnarray*}
for all $i,j=0,1,\ldots,r$, where we have used \eqref{three-term}.
Hence $\hat{\H}_r(x;\,\#\lambda)$ is a tri-diagonal matrix where at row $i$ the three elements 
are $(a_{i-1},b_i,a_{i+1})$. Therefore, $\hat{\H}_r(x;\#\lambda)$ is the $r$-truncation of the
 Jacobi matrix \eqref{jacobi}.
\end{proof}
 As a consequence we thus obtain:
 \begin{cor}
 \label{cor-jacobi}
 Let $\tau^\ell_r$ and $\tau^u_r$ be as in Theorem \ref{th1} and let 
 $J_r$ be the tri-diagonal matrix in Proposition \ref{prop1}. Then 
 $\tau^\ell_r=\lambda_{\min} (J_r)$ and $\tau^u_r=\lambda_{\max} (J_r)$. Therefore:
 \begin{equation}
 \label{cor-final}
 \lambda_{\min}(J_r)\,\downarrow\,\lf\mbox{ and } \lambda_{\max}(J_r)\,\uparrow\,\uf\mbox{ as $r\to\infty$.}
 \end{equation}
 Also for every $r\in\N$, $\tau^\ell_r$ (resp. $\tau^u_f$) is the smallest (resp. largest) root of the univariate 
 polynomial $T_{r+1}$. 
 \end{cor}
 \begin{proof}
  Follows from Theorem \ref{th1} and the definition of $\tau^\ell_r$ and $\tau^u_r$. The last statement can 
  be found in Dunkl and Xu \cite[Theorem 1.3.12]{dunkl}.
 \end{proof}

\begin{rem}
\label{sphere}
{\rm The use of an othonormal polynomial basis to reduce the 
initial $n$-dimensional problem \eqref{pb-initial-inf} to a standard 
($n$-dimensional) eigenvalue problem was already proposed in \cite{lass-siopt} and 
in de Klerk et al. \cite{las-3} and de Klerk and Laurent \cite{las-4}
but for the original $n$-dimensional problem and the reference measure $\lambda$
(and not on $\R$ for the pushforward measure $\#\lambda$).

In addition, in \cite{0} de Klerk and Laurent have used the univariate problem $\min\{ x: x\in [-1,1]\,\}$
as a building block to prove the $O(1/r^2)$ rate of convergence for the bounds
\eqref{pb-initial-inf} and \eqref{pb-initial-sup} in the multivariate case of the Sphere, the unit box and for some different reference measures $\lambda$.
They observed that if $f$ is the {\em univariate} polynomial $x$ then solving the 
resulting eigenvalue problem is computing the smallest eigenvalue of the Jacobi matrix associated with $\lambda$
(or equivalently, the smallest root of a certain orthogonal polynomial) as in Corollary \ref{cor-jacobi}. For specific 
reference measures $\lambda$, the associated  orthogonal polynomials have been 
well-studied (e.g. Chebyshev or Legendre polynomials), in particular 
the asymptotic behavior of their smallest (or largest) root used by the authors in \cite{0}. For more details the interested reader is referred to \cite{las-1,las-2,las-3,las-4,0}. However, in Corollary \ref{cor-jacobi} the underlying univariate problem
$\min\{\,x: x \in {\rm supp}(\#\lambda)\,\}$ in \eqref{equiv}
is {\em equivalent} to the original multivariate problem \eqref{pb-opt}. The price to pay is that the density of 
$\#\lambda$ is not known explicitly and makes the analysis of the rate of convergence more intricate.
}\end{rem}
\subsection*{Discussion}
As already mentioned, computing the entries of $\H_r(\#\lambda)$ (and hence of $\H(x;\,\#\lambda)$ as well) is easy but tedious for large $n$.
For a fixed $r\in\N$, and once the moment matrix $\H_r(\#\lambda)$ has been computed,
computing the scalar $\tau^u_r$ or $\tau^\ell_r$ in \eqref{relax-inf}-\eqref{relax-sup} is definitely easier than computing 
$\otheta_r$ or $\utheta_r$ as in the former one handles univariate moment matrices of size $r$ instead of
$n$-variate moment matrices of size ${n+d\choose n}$ in the latter. 
However we have not proved any rate of convergence
for $\tau^\ell_r\downarrow \lf$ whereas $\otheta_r\downarrow \lf$ at a rate $O(1/r^2)$  for some simple sets
$\om$ and appropriate measures $\lambda$. 
As mentioned earlier, convergence analysis of the sequence $(\tau^\ell_r)_{r\in\N}$ is difficult because
we do {\em not} have an explicit expression of the density of $\#\lambda$ w.r.t. Lebesgue measure
on $[\lf,\uf]$.

For illustration purpose, for $r=5,6$, we have considered four toy problems in $n=2$ variables
to compare the upper bounds $\tau^\ell_r$ on $\lf$ obtained in \eqref{relax-inf}
with the upper bounds $\otheta_r$ obtained in \eqref{pb-initial-inf} as
described in de Klerk and Laurent \cite{las-3}. Hence for the same $r$,
the former are obtained by solving eigenvalue problems with matrices of size $6$ for $r=5$
(resp. size $7$ for $r=6$) as opposed to matrices of size ${2+r\choose 2}=21$ for $r=5$ (resp. size $28$ for $r=6$)
for the latter.\\

\noindent
{\bf Motzkin polynomial:} $f(\x)=64\,(x_1^4x_2^2+x_1^2x_2^4)-48\,x_1^2x^2_2+1$\\
{\bf Matyas function:} $f(\x)=26\,(x_1^2+x_2^2)-48\,x_1x_2$\\
{\bf Booth function:} $f(\x)=(10\,x_1+20\,x_2-7)^2+(20\,x_1+10\,x_2-5)^2$\\
{\bf Three-hum-camel function:} $f(\x)=5^6x_1^6/6-1.05*5^4x_1^4+50\,x_1^2+245\,x_1x_2+25\,x_2^2$

\begin{table}[h] 
\begin{tabular}{|l|c|c|c|c|c|}
\hline
&&&&\\
\mbox{pb}& Moztkin & Matyas & Booth & Three-hump camel\\
&&&&\\
\hline 
$\otheta_5$&0.801 &3.69 & 69.81&9.58 \\
\hline 
$\tau^\ell_5$& 0.873&2.06 &56.64 &15.07 \\
\hline
&&&&\\
\hline 
$\otheta_6$& 0.801&2.99 &63.54 & 4.439\\
\hline 
$\tau_6^\ell$&0.808&1.68&45.49&12.68\\
\hline 
\end{tabular}
\caption{Comparing $\tau^\ell_r$ and $\otheta_r$ on 4 toy examples\label{table-1}}
\end{table}

In Table \ref{table-1} are displayed the results. Except for the Motzkin polynomial the bounds $\tau^\ell_r$ are comparable and even better
than the bounds $\otheta_r$.

\subsection*{Convergence rate}
After the initial version of this paper was released on a public repository, Laurent and Slot \cite{laurent-slot} have analyzed
the convergence $\tau^\ell_r\downarrow\lf$ as $r\to\infty$, of Theorem \ref{th1}. Quite remarkably
they prove a $O(\log^2r/r^2)$ rate of convergence under a weak geometric condition on $\om$ 
(satisfied for the ``simple" sets $\om$). Their proof is in the same vein and spirit as the one in Slot and Laurent \cite{las-5} where the authors prove convergence $\underline{\theta}_r\downarrow \lf$ at a
$O(\log^2r/r^2)$ rate for general convex bodies. For more details the interested reader is referred to \cite{laurent-slot}.


\section{Appendix}

\begin{lem}
\label{lem-aux}
Let $\H_r(x;\,\#\lambda)$ be the Hankel matrix associated with $\#\lambda$ and the polynomial $x\mapsto x$.
Then 
\begin{equation}
\label{lem-aux-1}
\lambda_{\min}(\H_r(x;\,\#\lambda),\H_r(\#\lambda))\,=\,\sup_a\,\{\,a:\:\H_r(x;\,\#\lambda)\,\succeq\,a\,\H_r(\#\lambda)\,\}.
\end{equation}
\begin{proof}
Let $c<\lf$ be arbitrary, fixed. As $x\geq \lf>c$ for all $x$ in the support of $\#\lambda$,
it follows that $\H_r(x-c;\,\#\lambda)\succ0$, and since
$\H_r(\#\lambda)\succ0$, 
\[\lambda_{\min}(\H_r(x-c;\,\#\lambda),\H_r(\#\lambda))\,=\,\sup_a\,\{\,a:\:\H_r(x-c;\,\#\lambda)\,\succeq\,a\,\H_r(\#\lambda)\,\}.\]
Notice also that since $\H_r(x-c;\,\#\lambda)=\H_r(x;\,\#\lambda)-c\,\H_r(\#\lambda)$:
\[\lambda_{\min}(\H_r(x-c;\,\#\lambda),\H_r(\#\lambda))\]
\begin{eqnarray*}
&=&\inf_a\,\{\,a:\:\exists p,\:\H_r(x-c;\,\#\lambda)\,p\,=\,a\,\H_r(\#\lambda)\,p\}\\
&=&\inf_a\,\{\,a:\:\exists p,\:\H_r(x;\,\#\lambda)\,p\,=\,(a+c)\,\H_r(\#\lambda)\,p\}\\
&=&-c+\inf_a\,\{\,a:\:\exists p,\:\H_r(x;\,\#\lambda)\,p\,=\,a\,\H_r(\#\lambda)\,p\}\\
&=&-c+\lambda_{\min}(\H_r(x;\,\#\lambda),\H_r(\#\lambda)).
\end{eqnarray*}
Next, 
\begin{eqnarray*}
-c+\lambda_{\min}(\H_r(x;\,\#\lambda),\H_r(\#\lambda))&=&
\lambda_{\min}(\H_r(x-c;\,\#\lambda),\H_r(\#\lambda))\\
&=&\sup_a\,\{a:\H_r(x;\,\#\lambda)\,\succeq\,(a+c)\,\H_r(\#\lambda)\}\\
&=&-c+\sup_a\,\{a:\H_r(x;\,\#\lambda)\,\succeq\,a\,\H_r(\#\lambda)\},
\end{eqnarray*}
and the proof is complete.
\end{proof}
\end{lem}

\section{Conclusion}

We have exhibited a strong (and perhaps surprising) connection between global optimization and 
spectral analysis of tri-diagonal univariate moment matrices  (equivalently,
roots of some sequence of univariate orthogonal polynomials).
Essentially computing the global minimum (resp. maximum) of a function
$f$ on a compact set $\om\subset\R^n$ reduces to a \emph{one-dimensional} problem, namely
computing the limit of the smallest
(resp. largest) eigenvalue of a sequence of tri-diagonal moment matrices whose size $r$
is independent of the dimension $n$. Of course the entries of 
these matrices require computing integrals $\int_\om f(x)^kd\lambda$,
$k\in\N$, for some choice of a measure $\lambda$ whose support is exactly 
$\om$. When $\om$ is a simple set then this is theoretically easy but becomes 
tedious for large $n$.  

On the theoretical side, the question of how fast such bounds converge to the global minimum
was not addressed in this paper but was addressed in the recent work of Laurent and Slot \cite{laurent-slot}.
In \cite{laurent-slot} the authors prove that
the bounds converge to the global minimum (resp. global maximum) at a
$O(\log^2r/r^2)$ rate, under a weak geometric condition on $\om$ (satisfied by all 
of simple sets except for the discrete hypercube).

On a more practical side, whether such an approach may provide 
good bounds for not too large $r$ and significant dimension $n$, is a topic of future research.
Interestingly, it is definitely related to a problem in numerical analysis, namely how efficiently 
can be computed integrals $\int_\om f(\x)^k\,d\lambda$ when $f$ is a polynomial, $\om$ is a simple set and 
$\lambda$ an appropriate measure whose support is $\om$.

At last but not least, observe that if the global optimum can be approximated 
at a $O(\log^2r/r^2)$ rate, so far one cannot approximate global minimizers
as we are essentially addressing a related \emph{one-dimensional} problem. Therefore 
our result in this paper suggests the following question: For the particular class of NP-hard 
polynomial optimization problems studied in this paper, is the problem of computing (only) the optimal value 
of same computational complexity as the problem of computing the optimal value \emph{and} global minimizers\,?

\end{document}